\documentclass[12pt]{article}
\usepackage{amsmath,amsfonts,amssymb,amsthm}

\newcommand{\GL}{\mathop{\mathrm{GL}}}
\newcommand{\Sp}{\mathop{\mathrm{Sp}}}
\newcommand{\GU}{\mathop{\mathrm{GU}}}
\newcommand{\Aut}{\mathop{\mathrm{Aut}}}
\newcommand{\Sym}{\mathop{\mathrm{Sym}}}
\newcommand{\GF}{\mathop{\mathrm{GF}}}
\newcommand{\AG}{\mathop{\mathrm{AG}}}
\newcommand{\tp}{\mathop{\mathrm{tp}}}
\newcommand{\SL}{\mathop{\mathrm{SL}}}

\def\DD{{\mathcal{D}}}
\newcommand{\XX}{\mathop{\mathcal{X}}}
\def\LL{{\mathcal{L}}}
\newcommand{\PP}{\mathop{\mathcal{P}}}
\def\BB{{\mathcal{B}}}

\def\Aut{{\rm Aut}}
\def\Sym{{\rm Sym}}
\def\calD{{\mathcal{D}}}
\def\calL{{\mathcal{L}}}
\def\calP{{\mathcal{P}}}
\def\calB{{\mathcal{B}}}

\def\De{\Delta}
\def\Om{\Omega}
\def\be{\beta}
\def\al{\alpha}
\def\ga{\gamma}
\def\de{\delta}
\def\si{\sigma}
\def\lam{\lambda}
\def\vp{\varphi}

\newtheorem{theorem}{Theorem}[section]
\newtheorem{lemma}[theorem]{Lemma}
\newtheorem{proposition}[theorem]{Proposition}

\newtheorem{construction}[theorem]{Construction}
\newtheorem{question}[theorem]{Question}

\newtheorem{unnumber}{}

\newtheorem{prepf}[unnumber]{Proof}
\newenvironment{pf}{\prepf\rm}{\endprepf}

\title{Constructing flag-transitive,\\ point-imprimitive designs}
\author{Peter J. Cameron\\Mathematical Institute\\University of St Andrews\\
North Haugh\\St Andrews, Fife, UK\\and\\
Cheryl E. Praeger\\School of Mathematics and Statistics\\University of
Western Australia\\Crawley, WA 6009, Australia}
\begin{document}
\maketitle

\begin{abstract}
We give a construction of a family of designs with a specified point-partition,
and determine the subgroup of automorphisms leaving invariant the 
point-partition. We give necessary and sufficient conditions for a design in 
the family to possess a flag-transitive group of automorphisms preserving the
specified point-partition. We give examples of flag-transitive designs
in the family, including a new symmetric $2$-$(1408,336,80)$ design with 
automorphism group $2^{12}:((3\cdot\mathrm{M}_{22}):2)$, and a construction
of one of the families of the symplectic designs (the designs $S^-(n)$)
exhibiting a flag-transitive, point-imprimitive automorphism group.
\end{abstract}

\section{Introduction}

In 1982, Sane~\cite{Sane} gave a construction of symmetric
$2$-$(q^2(q+2),q(q+1),q)$ designs from $q+2$ affine planes of order $q$. In
this paper we will give a generalisation of this construction. Our
construction is more general in three respects: first, we replace the
affine designs by general resolvable $2$-designs; second, we put a
symmetric design on the set indexing the resolvable designs; and third,
we use a transversal design to select the blocks.

The point sets of the designs we construct come with a natural partition.
A feature of our construction is that we are able to control the
group of automorphisms of the constructed design fixing this partition.
Thus, we can construct some new examples of
point-imprimitive, flag-transitive symmetric $2$-designs, including a
$2$-$(1408,336,80)$ design with automorphism group 
$2^{12}:((3\cdot\mathrm{M}_{22}):2)$. In addition, we show that the symmetric
designs $S^-(n)$ of symplectic type admit flag-transitive, point-imprimitive
groups of automorphisms which can be constructed using our technique.

We note that constructions of symmetric designs with special cases of the
parameters considered here, and related strongly regular graphs, have a long
history, starting with the work of Ahrens and Szekeres~\cite{as} in 1969.
Wallis~\cite{wallis} used affine designs to construct symmetric designs with
polarities, giving rise to strongly regular graphs; Fon-Der-Flaass obtained
more than exponentially many non-isomorphic strongly regular graphs by varying
this construction. Our interest is in the automorphism groups of the designs.

\section{A Sane-type construction}
\label{s:sane}

In this and the next section we present and discuss a construction for
symmetric $2$-designs based on a non-trivial point partition.
The general case of the construction, given in Section~\ref{s:gencon},
is a development of the following special construction, which is
similar to one given by Sharad S. Sane \cite{Sane}. 
Other developments of Sane's construction appear in the literature, see
for example, \cite{GhargeSane}.

\begin{construction}\label{con:qaff}{\rm
Let $\calD_0=(\De_0,\calL_0)$ be the desarguesian affine plane ${\rm AG}(2,q)$ of order $q$, and let $\calP_0=\{P_1,\dots,P_{q+1}\}$
denote its set of parallel classes of lines.
Let $\De_1:=\{1,\dots,q+2\}$. For each $\be\in\De_1$, let $\psi_\be:\calP_0\rightarrow \De_1\setminus\{\be\}$, be a bijection chosen in such a way that, for each $P\in\calP_0$ and each $\al\in\De_1$, there is a unique $\be$ such that $\psi_\be(P)=\al$. Moreover, 
for each $P\in\calP_0$, $\ga\in\De_0$, and $\be\in\De_1$, let $\ell_P(\ga)$ denote the line of $P$ containing $\ga$, and set $B(\ga,\be):=\{(\de,\psi_\be(P))\,|\,P\in\calP_0, \de\in \ell_P(\ga)\}$, or in other words,
\[
 B(\ga,\be)=\bigcup_{P\in\PP_0}\ell_P(\ga)\times\{\psi_\be(P)\}.
\]

We define the design $\calD=(\Om,\calB)$ to have point set
$\Om=\De_0\times\De_1$ and block set $\calB$ the disjoint union 
of the sets $\calB_\be=\{B(\ga,\be)\,|\,\ga\in \De_0\}$, for $\be\in\De_1$.
}
\end{construction}

The design $\calD$ has $|\De|=q^2(q+2)$ points, and $|\calB|=
q^2(q+2)$ blocks, each block having size $q(q+1)$. To see that $\calD$ is a
$2$-design we argue as follows.
If $(\de,\al), (\de',\al')\in\Om$ have distinct $\De_1$-entries (that is, if
$\al\ne\al')$), then there are exactly $q$ points
$\be\in\De_1\setminus\{\al,\al'\}$, and for each such $\be$, the
$\De_0$-points $\de,\de'$ lie in unique lines $\ell,\ell'$ of $P, P'$ respectively, where $\al=\psi_\be(P)$ and $\al'=\psi_\be(P')$. 
Since $P\ne P'$, the lines $\ell,\ell'$ intersect in a unique point, say $\ga(\be)$, and hence $(\de,\al), (\de',\al')\in\Om$
lie in $B(\ga(\be),\be)$. Thus these two points of $\Om$ lie in exactly $q$ common blocks. 
On the other hand if $(\de,\al), (\de',\al)\in\Om$ have the same 
$\De_1$-entry, then $\de,\de'$ lie in a unique line
$\ell\in\calL_0$ and hence $\de,\de'$ determine the unique 
parallel class $P\in\calP_0$ containing $\ell$. Also, there is a unique $\be$ such that $\psi_\be(P)=\al$. Therefore, there are
exactly $q$ blocks containing $(\de,\al), (\de',\al)$, namely
the sets $B(\ga,\be)$ such that $\ga\in\ell$.

Thus $\calD$ is a symmetric $2$-$(q^2(q+2),q(q+1),q)$ design.
Moreover, both the point set $\Om$ and the block set $\calB$
fall into $q+2$ classes, each of size $q$. Sane refers to these
designs as \emph{quasi-affine}.

For small values of
the parameter $q$, it turns out that the bijections $\psi_\be$
can be chosen so that $\calD$ admits a flag-transitive subgroup of
automorphisms preserving these partitions of the point set
and block set, namely for $q=2, 3, 4$, see \cite{p,r, lpr}.

\section{A general imprimitive design construction}
\label{s:gencon}

Here we present our design construction in which we extract, and generalise,
several critical features of the Sane-type construction given in
Section~\ref{s:sane}. Our objective is to construct a class of designs in which 
an identifiable sub-class are symmetric $2$-designs, and for which it is 
possible to determine those automorphisms that leave invariant the point
partition used in the construction.  We briefly describe three 
important features of the construction in Section~\ref{s:sane}. 

\begin{enumerate}
\item The first ingredient of our construction in Section~\ref{s:sane}
was an affine plane, and we will replace it here with a general resolvable 
$2$-design $\DD_0=(\De_0,\LL_0)$ with block set partitioned into a set $\PP_0$
of parallel classes. 

\item Next, the set of parallel classes of lines of 
the affine plane was identified in several different ways
with subsets of a second set $\De_1$, namely via bijections
$\psi_\be:\PP_0\rightarrow \De_1\setminus\{\be\}$, for each $\be\in\De_1$.
We generalise this by choosing, as the images of the maps, the blocks
of an arbitrary symmetric $2$-design $\DD_1$ with point set $\De_1$. (Note that 
the subsets $\De_1\setminus\{\be\}$, for  $\be\in\De_1$, form the blocks of a 
degenerate symmetric $2$-design on $\De_1$.)

\item Finally the
blocks in the construction in Section~\ref{s:sane} were determined by the maps 
$\psi_\be$ and by choosing, for each parallel class $P\in\PP_0$, 
a line in $P$. In Section~\ref{s:sane}, the choice 
was the line of $P$ containing a fixed point $\ga\in\De_0$, using the 
same point for each parallel class. For our general construction we also
require a method of selecting a block 
from each of the parallel classes in certain given subsets of $\PP_0$. 
For this we introduce
a transversal design $\DD_2=(\De_2,\LL_2)$ with point set $\De_2$ 
equal to the block set 
$\LL_0$ of $\DD_0$, and hence partitioned naturally into `groups', namely the
parallel classes. (The phrase \emph{transversal design} has more than one
meaning. We intend a design whose point set is partitioned into `groups' of
the same size; a block contains at most one point from each group; and two
points in different groups lie in a constant number of blocks.) 
Each block $\ga\in\LL_2$ has size at most $|\PP_0|$ and 
meets each group in at most one point; also two points in different groups
lie in a constant number of these blocks. 
A block of the transversal design $\DD_2$
can therefore be used to select blocks of $\DD_0$ from the parallel classes it meets. 
\end{enumerate}

Thus as input to our construction we have: 

\begin{enumerate}
\item[1.] a resolvable $2$-$(v_0,k_0,\lam_0)$ design $\calD_0=(\De_0,\calL_0)$ with $r$ parallel classes $\calP_0=\{P_1,\dots,P_r\}$ such that each $P_i$ consists of $s$  blocks, so
\begin{equation}\label{eq:d0}
v_0=sk_0,\quad (v_0-1)\lam_0=r(k_0-1);
\end{equation}
\item[2.] a symmetric $2$-$(v_1,r,\lam_1)$ design $\calD_1=(\De_1,\calL_1)$ with blocks of size $r$, so
\begin{equation}\label{eq:d1}
(v_1-1)\lam_1=r(r-1);
\end{equation}
together with bijections $(\psi_\be)_{\be\in\calL_1}$, where  $\psi_\be:\calP_0\rightarrow \be$ such that,
for each $P\in\calP_0$ and $i\in\De_1$, there is a unique $\be\in\calL_1$ with $\psi_\be(P)=i$.
\item[3.] a transversal design $\calD_2=(\De_2,\calL_2)$, where the point set $\De_2$ consists of $r$ groups of size $s$ and we identify $\De_2$ with the union $\cup_{P\in\calP_0}P$; each
block has size $k_2\leq r$ and meets each group in at most one point; and two points in different groups lie in exactly $\lam_2$ blocks. So there are $b_2$ blocks and each point lies in $r_2$ blocks, where
\begin{equation}\label{eq:d2}
r_2=\frac{(r-1)s\lam_2}{k_2-1},\quad b_2=\frac{r(r-1)s^2\lam_2}{k_2(k_2-1)};
\end{equation}
\end{enumerate}

We now give the construction, and then find necessary and sufficient
conditions on the parameters for it to be a symmetric $2$-design. Note that,
by the definition of $\DD_2$, for each $\ga\in\LL_2$,  $\be\in\LL_1$ and
$j\in\be$, the intersection $\ga\cap\psi_\be^{-1}(j)$ is either empty or a
single  block of the parallel class $\psi_\be^{-1}(j)$.

\begin{construction}\label{con:gencon} {\rm Let
$\calD_0, \calD_1, \calD_2$ and $(\psi_\be)_{\be\in\calL_1}$
be as above. For each  $\ga\in\calL_2$ and
$\be\in\calL_1$, let
\[
B_\be(\ga):=\bigcup_{j\in\be}((\ga\cap \psi_\be^{-1}(j))\times\{j\}).
\]
Define the design $\calD=(\De_0\times\De_1,\calB)$ to have point set
$\De_0\times\De_1$ and block set $\calB$, the disjoint union 
of the sets $\calB_\be=\{B_\be(\ga)\,|\,\ga\in \LL_2\}$, for $\be\in\calL_1$.
}
\end{construction}

\begin{theorem}\label{thm:gencon}
Let $\calD$ be the design of Construction~{\rm\ref{con:gencon}}.
\begin{enumerate}
\item[(a)] Then $\calD$ is a 
$1$-$(v_0v_1,k_0k_2,rr_2)$-design with $v_1b_2$ blocks, where the parameters 
are as in {\rm(\ref{eq:d0}), (\ref{eq:d1}), (\ref{eq:d2})}. 
\item[(b)] $\calD$ is symmetric if and only if
\[
r(r-1)s\lam_2=k_0k_2(k_2-1)
\]
\item[(c)] $\calD$ is a $2$-design if and only if
\[
\lam_1=\lam_0\frac{(r-1)s}{k_2-1}
\]
and in this case it is a $2$-$(v_0v_1,k_0k_2,\lam_1\lam_2)$ design.
\end{enumerate}
\end{theorem}

\begin{pf}
There are $|\De_0\times\De_1|=v_0v_1$ points and $|\calB|=v_1b_2$ blocks, with each block of size $k_0k_2$. Consider a point 
$(\de,i)\in\De_0\times\De_1$. The point $i$ of $\calD_1$ lies in $r$ blocks $\be\in\calL_1$. For each of these blocks $\be$, the
parallel class $P:=\psi_\be^{-1}(i)$ is uniquely determined, and
hence also the block $\ell\in P$ that contains $\de$ is determined.
Since $\ell$ belongs to $r_2$ blocks of $\calD_2$, it follows that
$(\de,i)$ lies in exactly $rr_2$ blocks of $\calD$. This proves part (a). Further, $\calD$ is symmetric if and only if $v_0v_1=
v_1b_2$, that is,
\[
sk_0v_1=v_1\frac{r(r-1)s^2\lam_2}{k_2(k_2-1)}
\]
and part (b) follows.

If $(\de,i), (\de',i')\in\De_0\times\De_1$ have distinct $\De_1$-entries,
then there are exactly $\lam_1$ blocks 
$\be\in\calL_1$ containing $\{i,i'\}$, and for each such $\be$, 
the points $\de,\de'$ lie in unique blocks $\ell,\ell'$ of $P, P'$ respectively, where $i=\psi_\be(P)$ and $i'=\psi_\be(P')$. 
Since $P\ne P'$, $\ell,\ell'$ belong to exactly
$\lam_2$ blocks of $\calL_2$, and for each of these blocks $\ga\in\calL_2$, the two points $(\de,i), (\de',i')$
lie in $B_\be(\ga)$. Thus these two points lie in exactly $\lam_1\lam_2$ blocks of $\calD$. 

On the other hand if $(\de,i), (\de',i)\in\De_0\times\De_1$ have the same 
$\De_1$-entry, then $\de,\de'$ lie in $\lam_0$ blocks of $\calD_0$.
Each of these blocks $\ell\in\calL_0$ lies in a unique
parallel class $P(\ell)\in\calP_0$, and there is a unique $\be(\ell)$ such that $\psi_{\be(\ell)}(P(\ell))=i$. There are
exactly $r_2$ blocks of $\calL_2$ containing $\ell$, and hence, for each of the $\lam_0$ blocks $\ell$, 
there are $r_2$ blocks of $\calD$ containing $(\de,i), (\de',i)$, namely the sets $B_{\be(\ell)}(\ga)$ such that $\ell\in\ga$. Thus these two points lie in exactly $\lam_0r_2$ blocks of $\calD$.
This implies that $\calD$ is a $2$-$(v_0v_1,k_0k_2,\lam)$
design if and only if 
\[
\lam=\lam_1\lam_2=\lam_0r_2=\lam_0\lam_2\frac{(r-1)s}{k_2-1}.
\]
Thus (c) is proved.
\end{pf}

\section{A special case}

All the constructions given in Sections~\ref{s:96}, \ref{s:symp}
and \ref{s:sporadic} have the feature that
the design $\mathcal{D}_1$ is a trivial $2$-$(v_1,v_1-1,v_1-2)$ design.
In this case, there is an unexpected connection with Latin squares.

Recall the bijections $\psi_\beta$ in our construction: $\psi_\beta$ maps
the set $\mathcal{P}_0$, of cardinality $r$, to the block $\beta$ of the
design $\mathcal{D}_1$; and we require that for each $P\in\calP_0$ and
$i\in\De_1$, there is a unique $\be\in\calL_1$ with $\psi_\be(P)=i$.

Suppose that $k_1=r=v_1-1$. Add a `dummy
point' $\infty$ to $\mathcal{P}_0$. Each block $\beta$ is the complement of
a single point $b$; extend the bijection $\psi_\beta$ to map $\infty$ to
$b$. Now the domain and range of each $\psi_\beta$ has cardinality $v_1$;
each map is a bijection; and our extra condition shows that there is a unique
bijection mapping any element of $\mathcal{P}_0\cup\{\infty\}$ to any element
of $\Delta_1$. So, numbering the elements of $\mathcal{P}_0\cup\{\infty\}$
and of $\Delta_1$ from $1$ to $v_1$, the $\psi_\beta$ are the rows of a Latin
square.

Conversely, any Latin square gives a collection of bijections
satisfying our condition. Hence, in the case where $r=v_1-1$, given the
designs $\mathcal{D}_i$ for $i=0,1,2$, any Latin square gives a design,
which will be symmetric (resp., a $2$-design) if the conditions of
Theorem~\ref{thm:gencon} are satisfied.

\medskip

All our constructions have a further important feature, which simplifies
the general method even further.

A resolvable design is \emph{affine} if any two blocks in different 
resolution classes intersect in a constant number $\mu$ of points. In
an affine resolvable $2$-design with $r$ resolution classes each containing
$s$ blocks, we have $v=s^2\mu$, $k=s\mu$, and $r=(s^2\mu-1)/(s-1)$.

\begin{theorem}
If there exists an affine resolvable $2$-design with $r$ resolution classes
each containing $s$ blocks, in which blocks in different classes intersect
in $\mu$ points, then there exists a symmetric $2$-$(s^2\mu(r+1), 
s\mu r, \mu(r-1))$ design.
\end{theorem}

\begin{proof}
We take $\mathcal{D}_0$ to be the given affine design, and $\mathcal{D}_2$
to be its dual. The groups in $\mathcal{D}_2$ correspond to the resolution
classes in $\mathcal{D}_0$, and every block of $\mathcal{D}_2$ meets every
group (because every point of $\mathcal{D}_0$ lies on a block in every 
resolution class). We take $\mathcal{D}_1$ to be the trivial symmetric
$2$-$(r+1,r,r-1)$ design, and the bijections $\psi_\beta$ to come from a
Latin square of order $r+1$.

We have $k_2=r$, and so
\[r(r-1)s\lambda_2=r(r-1)s\mu=k_2(k_2-1)k_0\]
and
\[\lambda_0(r-1)s/(k_2-1)=(s\mu-1)s/(s-1)=(s^2\mu-1)/(s-1)-1=r-1=\lambda_1,\]
so $\mathcal{D}$ is a symmetric $2$-design, by parts (b) and (c) of
Theorem~\ref{thm:gencon}.
\end{proof}

\paragraph{Remark} Since the number of different Latin squares is very large,
this construction should produce many non-isomorphic symmetric designs. Most
of these will not be flag-transitive.  In the next section we study automorphisms of the designs from Construction~\ref{con:gencon}, and in Theorem~\ref{flagtr} 
we give necessary and sufficient conditions for the automorphism group preserving 
the partition to act flag-transitively.

\section{Automorphisms of the designs}

In this section we determine the subgroup $G$ of automorphisms of the design
$\DD$ in  Construction~\ref{con:gencon} that leaves invariant the partition
\[
\XX:=\{\De_0\times\{i\}\,|\ i\in\De_1\}
\]
of the point set $\De_0\times\De_1$ of $\DD$. Thus
\[
G:=\Aut(\DD)\cap (\Sym(\De_0)\wr\Sym(\De_1))
\]
and a typical element $g$ of $G$ has the form 
\[ 
g=(h_1,\dots,h_{v_1})\si
\]
where the $h_i\in\Sym(\De_0)$ and $\si\in\Sym(\De_1)$. We note that this may
not be the full automorphism group of the design: this is the case for the
point-imprimitive groups of the symplectic designs described in
Section~\ref{s:symp}.

First we  
define two properties of elements $g=(h_1,\dots,h_{v_1})\si\in \Sym(\De_0)
\wr\Sym(\De_1)$, and prove that the properties hold for elements of $G$. 
The first property is:
\begin{equation}\label{prop1}
\forall\ \be\in\LL_1,\ \forall\ j\in\be,\quad (\psi_\be^{-1}(j))^{h_j}
=\psi_{\be^\si}^{-1}(j^\si)
\end{equation}
or in other words $\psi_\be(P)=j \iff\ 
\psi_{\be^\si}(P^{h_j})=j^\si$ for any $\be\in\LL_1, j\in\be$ and $P\in\PP_0$.
The second property involves blocks $\be\in\LL_1, 
\ga\in\LL_2$, and the following set:
\begin{equation}\label{gadash}
\ga' = \ga'_{\ga,\be,g}:=\{ (\ga\cap\psi_\be^{-1}(j))^{h_j}\ |\,j\in\be\ 
\mbox{such that}\ \ga\cap\psi_\be^{-1}(j)\ne\emptyset\}.
\end{equation}
The property is: $\forall\ \be\in\LL_1,\ \forall\ \ga\in\LL_2,$
\begin{equation}\label{prop2}
\ga'=\ga'_{\ga,\be,g}\in\LL_2,\ 
\mbox{and moreover}\ \forall\ j\in\be,\ \ga'\cap\psi_{\be^\si}^{-1}(j^\si)=
(\ga\cap \psi_\be^{-1}(j))^{h_j}.
\end{equation}

\begin{theorem}\label{hasprops}
Let $\calD=(\De_0\times\De_1,\calB)$ be the design of Construction~{\rm\ref{con:gencon}}, and let 
$g=(h_1,\dots,h_{v_1})\si\in \Aut(\DD)\cap\left(\Sym(\De_0)\wr\Sym(\De_1)\right)$.
Then the following hold.
\begin{enumerate}
\item $\si\in\Aut(\DD_1)$, and, for each $j\in\De_1$, $h_j$ lies in the subgroup $\Aut^*(\DD_0)$ 
of $\Aut(\DD_0)$ preserving the set $\PP_0$ of parallel classes.
\item  Properties {\rm(\ref{prop1})} and {\rm(\ref{prop2})} hold for $g$.
\item  For each $\be\in\LL_1, \ga\in\LL_2$, $B_\be(\ga)^g=B_{\be^\si}(\ga')$
with $\ga'=\ga'_{\ga,\be,g}$ as in {\rm(\ref{gadash})}, and in particular,
$G$ leaves invariant the partition $\{\BB_\be\,|\,\be\in\LL_1\}$ of $\BB$.
\item If $\si=1$ then, for each $j\in\De_1$, $h_j$ lies in the subgroup $T_0$ 
of $\Aut(\DD_0)$ fixing setwise each of the parallel classes in $\PP_0$.
\end{enumerate}
\end{theorem}

\begin{proof}
Suppose that $g=(h_1,\dots,h_{v_1})\si\in \Aut(\DD)$. Let
$\be\in\LL_1, \ga\in\LL_2$. Then the image under $g$ of the block $B_\be(\ga)$ 
of $\DD$ is also a block of $\DD$, say $B_{\be'}(\ga')$ where $\be'\in\LL_1, 
\ga'\in\LL_2$. By definition, 
$B_\be(\ga)=\bigcup_{j\in\be}((\ga\cap\psi_\be^{-1}(j))\times\{j\})$. Note that
for some $j\in\be$ the intersection $\ga\cap\psi_\be^{-1}(j)$ may be empty.

Now $(\de,j)\in B_\be(\ga)$ if and only if $j\in\be$ and, for 
$P:=\psi_\be^{-1}(j)$, the point $\de$ lies in the unique part $\ga\cap P$ 
of $P$. This holds if and only if $(\de,j)^g=
(\de^{h_j},j^\si)\in B_{\be'}(\ga')$, and hence if and only if $j^\si\in\be'$ 
and for $P':=\psi_{\be'}^{-1}(j^\si)$, $\de^{h_j}\in \ga'\cap P'$. 
We make several deductions from this. 

First we deduce that, for each $j\in\be$,
$(\ga\cap\psi_\be^{-1}(j))^{h_j}=\ga'\cap\psi_{\be'}^{-1}(j^\si)$.
Next, if we fix $\be$ and let 
$j$ vary over the points of $\be$, we see that
$\be'=\{j^\si\,|\,j\in\be\}=\be^\si$. In view of the
previous sentence, this means that   
$\ga'$ is the set $\ga'_{\ga,\be,g}$ defined in 
(\ref{gadash}), and hence $\ga'_{\ga,\be,g}\in\LL_2$. 
This proves property (\ref{prop2}) for $g$, and also that
$B_\be(\ga)^g=B_{\be^\si}(\ga'_{\ga,\be,g})$. 
If we now let $\be$ vary 
over $\LL_1$, we deduce that $\si$ leaves invariant the set $\LL_1$, and hence 
$\si\in\Aut(\DD_1)$, and 
$G$ leaves invariant the partition $\{\BB_\be\,|\,\be\in\LL_1\}$ of $\BB$.

Next fix $\be\in\LL_1$ and $j\in\be$, and set $P:=\psi_\be^{-1}(j)$ and
$P':=\psi_{\be^\si}^{-1}(j^\si)$, so $P, P'\in\PP_0$. 
By the definition of $\DD_2$, for each block $\ell\in P$ there exists 
$\ga\in\LL_2$ such that $\ga\cap P=\{\ell\}$.
We have proved above that $P'$ contains  
\[
\ga'_{\ga,\be,g}\cap\psi_{\be^\si}^{-1}(j^\si)=(\ga\cap P)^{h_j}=\{\ell^{h_j}\}
 \]
and therefore $P'$ contains $\ell^{h_j}$ for each $\ell\in P$.
Since $P$ and $P'$ are partitions of $\De_0$ and $h_j$ is a bijection, 
this implies that 
$P'=P^{h_j}$, and so property (\ref{prop1}) is proved for $g$. 

Finally, fixing $j\in\De_1$, and choosing $\ell\in\LL_0$, it follows from
the properties of $\PP_0$ and the $(\psi_\be)_{\be\in\LL_1}$
that there is a unique $P\in\PP_0$
containing $\ell$, and a unique $\be\in\LL_1$ such that $\psi_\be(P)=j$. We have
just shown that $\ell^{h_j}$ lies in $P^{h_j}\in\PP_0$. Hence $h_j$ leaves invariant both the
block set $\LL_0$ of $\DD_0$, and the set $\PP_0$ of parallel classes. So $h_j\in\Aut^*(\DD_0)$. 
Moreover, if $\si=1$, then  $P^{h_j}=\psi_{\be^\si}^{-1}(j^\si)=
\psi_\be^{-1}(j)=P$, and this holds for all $j$ and for all $P$. Thus in this case each $h_j\in T_0$.
\end{proof}

It is not difficult to prove a converse to this lemma, showing that property
(\ref{prop2}) characterises membership of $G$. 

\begin{lemma}\label{lem:charg}
The group $G=\Aut(\DD)\cap (\Sym(\De_0)\wr\Sym(\De_1))$ is equal to the set 
of all elements $g=(h_1,\dots,h_{v_1})\si$ such that 
\[
\mbox{each}\ h_j\in\Aut^*(\DD_0), \si\in\Aut(\DD_1),\
\mbox{and {\rm(\ref{prop2})} holds for $g$}.
\]
\end{lemma}

\begin{proof}
Suppose that $g=(h_1,\dots,h_{v_1})\si\in \Aut^*(\DD_0)\wr\Aut(\DD_1)$ and that
property  {\rm(\ref{prop2})} holds for $g$.
Let $\be\in\LL_1,\ \ga\in\LL_2$ and consider the image of $B_\be(\ga)$ under $g$.
Now $B_\be(\ga)=\cup_{j\in\be}((\ga\cap\psi_\be^{-1}(j))\times\{j\})$, and by {\rm(\ref{prop2})},
for each $j\in\be$, 
\[
((\ga\cap\psi_\be^{-1}(j))\times\{j\})^g=(\ga\cap\psi_\be^{-1}(j))^{h_j}\times\{j^\si\}
=(\ga'_{\ga,\be,g}\cap\psi_{\be^\si}^{-1}(j^\si))\times\{j^\si\}
\]
and also $\ga'_{\ga,\be,g}\in\LL_2$. It follows that $B_\be(\ga)^g=
B_{\be^\si}(\ga'_{\ga,\be,g})
\in\LL$, and therefore $g\in\Aut(\DD)$, whence $g\in G$.
The converse follows from Lemma~\ref{hasprops}.
\end{proof}

\bigskip
Since automorphisms of designs are defined as permutations of the point set 
with the property that the block set is left invariant, the arguments above 
exhibit clearly the relationship between $\DD$ and the input designs $\DD_0$
and $\DD_1$ whose point sets are Cartesian factors of the point set of $\DD$.
However the role of the transversal design $\DD_2$ needs further clarification.
By Lemma~\ref{hasprops}.3, $G$ leaves invariant the
partition $\{\BB_\be\,|\,\be\in\LL_1\}$, where $\BB_\be=\{B_\be(\ga)
\,|\,\ga\in\LL_2\}$, and permutes the parts $\BB_\be$ of this 
partition in the same way that $G$ permutes $\LL_1$.  
We show further that, for $\be\in\LL_1$,
the setwise stabiliser $G_\be$ induces a group 
of automorphisms of $\DD_2$. Recall that the point set $\De_2$ of $\DD_2$ 
is equal to $\cup_{P\in\PP_0}P$. For $\ga\in\LL_2$, let $\PP_0^{(\ga)}
:=\{P\in\PP_0\,|\,\ga\cap 
P\ne\emptyset\}$.

\begin{lemma}\label{d2}
For $\be\in\LL_1$, define the map $\vp_\be:G_\be\rightarrow \Sym(\De_2)$ as 
follows. For $g=(h_1,\dots,h_{v_1})\si\in G_\be$ and $\ell\in\De_2$,
\[
(g)\vp_\be: \ell\mapsto \ell^{h_j},\ \mbox{where $j$ is the unique element of $\be$ 
such that}\ \ell\in\psi_\be^{-1}(j).
\]
Then 
\begin{enumerate}
\item $\vp_\be$ is a homomorphism and $(G_\be)\phi_\be
\leq\Aut(\DD_2)$;
\item  $(g)\vp_\be:\ga\rightarrow\ga'_{\ga,\be,g}$, and
$((\ga\cap\psi_\be^{-1}(j))\times\{j\})^g=(\ga'_{\ga,\be,g}\cap
\psi_\be^{-1}(j^\si))\times\{j^\si\}$ whenever $\psi_\be^{-1}(j)\in
\PP_0^{(\ga)}$;
\item $(g)\vp_\be$ fixes $\ga$
if and only if $g\in G_{B_\be(\ga)}$.
\end{enumerate}
\end{lemma}

\begin{proof}
First note that, for a given $\ell$, there is a unique $P\in\PP_0$ such that
$\ell\in P$, and $j=\psi_\be(P)$ is the unique element of $\be$ such that
$\psi_\be^{-1}(j)=P$. Hence $j$ is the unique element of $\be$ such that
$\psi_\be^{-1}(j)$ contains $\ell$.
By Lemma~\ref{hasprops}.1, $h_j\in\Aut^*(\DD_0)$, and hence $(\ell)
(g)\vp_\be=\ell^{h_j}$ is an element of $\De_2$. 
Thus $(g)\vp_\be$ is a well-defined map $\De_2\rightarrow \De_2$.

To prove that $(g)\vp_\be$ is one-to-one, suppose that $(\ell_1)(g)\vp_\be
=(\ell_2)(g)\vp_\be$, where $\ell_1,\ell_2\in\De_2$. 
Then, for $i=1,2$, $\ell_i\in P_i:=\psi_\be^{-1}(j_i)$ for a unique 
$j_i\in\be$, and 
\[
\ell_1^{h_{j_1}}=(\ell_1)(g)\vp_\be=(\ell_2)(g)\vp_\be=\ell_2^{h_{j_2}}\in P
\]
for some $P\in\PP_0$. This implies that, for $i=1,2$, $P=P_i^{h_{j_i}}$, or
equivalently, $P=(\psi_\be^{-1}(j_i))^{h_{j_i}}$ which by (\ref{prop1}) is 
equal to $\psi_\be^{-1}(j_i^\si)$ since $\be^\si=\be$. Since $\psi_\be$ is
a bijection this means that $j_1^\si=j_2^\si$, and since $\si$ is bijective,
$j_1=j_2$, so $P_1=P_2$. Then, since $\ell_1^{h_{j_1}}=\ell_2^{h_{j_2}}=
\ell_2^{h_{j_1}}$ and $h_{j_1}$ acts faithfully on $\cup_{P'\in\PP_0}P'$,
it follows that $\ell_1=\ell_2$. Thus $(g)\vp_\be$ is one-to-one, and hence is
a permutation of $\De_2$ since $\De_2$ is finite.
 
Next we prove that $(g)\vp_\be$ leaves $\LL_2$ invariant, and hence lies 
in $\Aut(\DD_2)$. Let $\ga\in\LL_2$.  
Then $\ga$ is the set of all $\ell_P:=\ga\cap P$, 
for $P\in\PP_0^{(\ga)}$. By the definition of $\vp_\be$, 
for each $P\in\PP_0^{(\ga)}$, $(\ell_P)(g)\vp_\be
=\ell_P^{h_j}$ where
$j=\psi_\be(P)$, and by (\ref{prop2}) and since $\be^\si=\be$, 
$\ell_P^{h_j}=(\ga\cap\psi_\be^{-1}(j))^{h_j}=\ga'_{\ga,\be,g}\cap
\psi_\be^{-1}(j^\si)$ and $(\ell_P\times\{j\})^g=\ell_P^{h_j}\times\{j^\si\}$. 
Thus $(\ga)(g)\vp_\be$ is the set of 
$\ga'_{\ga,\be,g}\cap\psi_\be^{-1}(j^\si)$ for all $j\in \psi_\be
(\PP_0^{(\ga)})$, and since $\ga'_{\ga,\be,g}\in\LL_2$ and has the same 
size as $\ga$, it follows that 
$(\ga)(g)\vp_\be=\ga'_{\ga,\be,g}\in\LL_2$. Thus part (2) is proved.

In particular, 
$(g)\vp_\be$ fixes $\ga$, that is to say, $\ga'_{\ga,\be,g}=\ga$,
if and only if $\si$ leaves $\PP_0^{(\ga)}$ invariant and 
$(\ga\cap\psi_\be^{-1}(j))^{h_j}=\ga\cap\psi_\be^{-1}(j^\si)$ 
whenever $\psi_\be^{-1}(j)\in\PP_0^{(\ga)}$. This is equivalent to
the property that $g$ permutes the set
$\{(\ga\cap\psi_\be^{-1}(j))\times\{j\}\,|\,\psi_\be^{-1}(j)\in
\PP_0^{(\ga)}\}$,
which, in turn, is equivalent to $g\in G_{B_\be(\ga)}$. Thus part (3)
is proved.

Finally we prove that $\vp_\be$ is a homomorphism. Let $g, x\in G_\be$ 
with $g$ as in the statement and $x=(y_1,\dots,y_{v_1})\tau$. Let $\ell, 
P, j$ be as in the first paragraph. Then $((\ell)(g)\vp_\be)
(x)\vp_\be=(\ell^{h_j})(x)\vp_\be=(\ell^{h_j})^{y_k}$ where $\ell^{h_j}\in
\psi_\be^{-1}(k)$. Now $\ell\in\psi_\be^{-1}(j)$, so $\ell^{h_j}\in 
(\psi_\be^{-1}(j))^{h_j}$ which by (\ref{prop1}) is equal to 
$\psi_\be^{-1}(j^\si)$. Thus $k=j^\si$ and so $((\ell)(g)\vp_\be)
(x)\vp_\be=\ell^{h_jy_{j\si}}$.
Since $gx=(h_1y_{1\si^{-1}},\dots,h_{v_1}y_{v_1\si^{-1}})
\si\tau$, it follows that $((\ell)(g)\vp_\be)
(x)\vp_\be=(\ell)(gx)\vp_\be$, and since this holds for all $\ell$ we 
have $(g)\vp_\be(x)\vp_\be=(gx)\vp_\be$, so $\vp_\be$ is a homomorphism. 
\end{proof}

Let $\pi:G\rightarrow\Sym(\De_1)$ and, for $1\leq j\leq v_1$ let $\pi_j:
G_j\rightarrow \Sym(\De_0)$ be defined by $(g)\pi=\si$ and $(g)\pi_j=h_j$
for $g=(h_1,\dots,h_{v_1})\si$ in $G$ or $G_j$ respectively.

\begin{theorem}\label{flagtr}
The group $G$ is flag transitive on $\DD$ if and only if the following four conditions all hold.
\begin{enumerate} 
\item $(G)\pi$ is flag-transitive on $\DD_1$; 
\item for each $j\in\De_1$, $(G_j)\pi_j$ is flag-transitive
on $\DD_0$;
\item for each $\be\in\LL_1$, $(G_\be)\vp_\be$ is flag-transitive 
on $\DD_2$;
\item for $\be\in\LL_1, j\in\be, \ga\in\LL_2$ such that $\ga\cap\psi_\be^{-1}(j)\ne\emptyset$, the group $(G_{\be,\ga,j})\pi_j$ is transitive on the block
$\ga\cap\psi_\be^{-1}(j)\in\LL_0$.
\end{enumerate}
\end{theorem}

\begin{proof}
Suppose first that $G$ is flag-transitive on $\DD$. Then in particular $G$ is 
transitive on the point set $\De_0\times\De_1$, so $(G)\pi$ is transitive on 
$\De_1$ and for each $j\in\De_1$, $(G_j)\pi_j$ is transitive on $\De_0$. Let 
$\de\in\De_0$. Then $G_{(\de,j)}=G_{j,\de}$ is transitive on the set of blocks 
that contain $(\de,j)$. 
By definition of $\DD$ these are the blocks $B_{\be}(\ga)$ such 
that $j\in\be$ and $\de\in \ga\cap\psi_\be^{-1}(j)$. Now for each $\be\in
\LL_1$ containing $j$, $\psi_\be^{-1}(j)\in\PP_0$ contains a unique block 
$\ell\in \LL_0$ containing $\de$, and  there exists $\ga\in\LL_2$ such that 
$\ga\cap P=\ell$. Thus $B_\be(\ga)$ contains $(\de,j)$ and it follows that
$G_{j,\de}$ and hence also $G_j$, acts transitively on the set of blocks $\be
\in\LL_1$ containing $(\de,j)$. Thus $(G)\pi$ is flag-transitive on $\DD_1$.

Further, each $\ell\in\LL_0$
containing $\de$ lies in a unique parallel class $P\in\PP_0$, and by the 
properties of the $(\psi_\be)_{\be\in\LL_1}$, there is a unique $\be\in\LL_1$
such that $\psi_\be(P)=j$. Also, as remarked above, there exists 
$\ga\in\LL_2$ such that $\ga\cap P
=\ell$. For this $\be$, and any such $\ga$, the block $B_\be(\ga)$ contains
$(\ga\cap\psi_\be^{-1}(j))\times\{j\}=\ell\times\{j\}$ which contains
$(\de,j)$. This implies that $G_{j,\de}$ is transitive on the set of blocks 
$\ell\in\LL_0$ containing $\de$, and hence that $(G_j)\pi_j$ is 
flag-transitive on $\DD_0$.

To examine the induced action on $\DD_2$, let $\be\in\LL_1$. By 
Lemma~\ref{hasprops}.3, $G$ permutes the sets $\BB_\be$ in the same way 
that it permutes $\LL_1$, and hence $G_\be$ is
transitive on $\BB_\be$. The stabiliser in $G_\be$ of $B_\be(\ga)\in\BB_\be$
is equal to $G_{B_\be(\ga)}$ and hence is transitive on the points 
$(\de,j)$ of $B_\be(\ga)$. This implies that $G_{B_\be(\ga)}$ is transitive 
on the set $\{(\ga\cap\psi_\be^{-1}(j))\times\{j\}\,|\,
\ga\cap\psi_\be^{-1}(j)\ne\emptyset\}$. 
By Lemma~\ref{d2}, $(G_{B_\be(\ga)})\vp_\be$ is the 
stabiliser in $(G_\be)\vp_\be$ of the block $\ga\in\LL_2$, and therefore
$(G_\be)\vp_\be$ is transitive on the points  
of $\DD_2$ incident with $\ga$, namely $\ga\cap
\psi_\be^{-1}(j)$ where $\psi_\be^{-1}(j)\in\PP_0^{(\ga)}$. 
Thus $(G_\be)\vp_\be$ is
flag-transitive on $\DD_2$. Also the stabiliser $G_{\be,\ga,j}$ in 
$G_{B_\be(\ga)}$ of the subset $(\ga\cap\psi_\be^{-1}(j))\times\{j\}$ is
transitive on it, and hence $(G_{\be,\ga,j})\pi_j$ is transitive on
$\ga\cap\psi_\be^{-1}(j)$. Thus if $G$ is flag-transitive on $\DD$ then
properties 1--4 all hold.

Conversely suppose that the four properties hold for $G$, and consider two 
flags $F=((\de,j),B_\be(\ga))$ and $F'=((\de',j'),B_{\be'}(\ga'))$ of $\DD$.
Since $(G)\pi$ is transitive on the blocks of $\DD_1$ we may replace $F'$ by
its image under some element of $G$ if necessary and assume that $\be'=\be$.
Next since $(G_\be)\vp_\be$ is transitive on the blocks of $\DD_2$, we may 
further assume that $\ga'=\ga$. Then, since $F, F'$ are flags, we have
$(\de,j)\in (\ga\cap\psi_\be^{-1}(j))\times\{j\}$ and
$(\de',j')\in (\ga\cap\psi_\be^{-1}(j'))\times\{j'\}$. 
Since $(G_\be)\vp_\be$ is flag-transitive on
$\DD_2$, $((G_\be)\vp_\be)_\ga$, which is equal to $(G_{B_\be(\ga)})\vp_\be$
by Lemma~\ref{d2}, is transitive on the points of $\DD_2$ incident with 
$\ga$, and hence some $g\in G_{B_\be(\ga)}$ maps $(\ga\cap\psi_\be^{-1}(j'))
\times\{j'\}$ to  $(\ga\cap\psi_\be^{-1}(j))\times\{j\}$. Thus we may assume 
further that $j'=j$ so that $\de,\de'\in\ga\cap\psi_\be^{-1}(j)$. Then since
$(G_{\be,\ga,j})\pi_j$ is transitive on $\ga\cap\psi_\be^{-1}(j)$, some 
element of $(G_{\be,\ga,j})\pi_j$ maps $F'$ to $F$. Thus $G$ is 
flag-transitive on $\DD$. 
\end{proof}

\section{The $2$-$(96,20,4)$ designs revisited}\label{s:96}

The paper \cite{lpr} describes four flag-transitive symmetric $2$-$(96,20,4)$
designs (three of which were previously known), and shows
that these are the only possible examples of such designs.

Our construction can potentially be applied to these designs, with
\begin{itemize}
\item $\mathcal{D}_0$ is the unique $2$-$(16,4,1)$ design, the affine plane
of order $4$ (which has a unique resolution), cf.~\cite[Lemma 3.2]{lpr};
\item $\mathcal{D}_1$ is the trivial $2$-$(6,5,4)$ design;
\item $\mathcal{D}_2$ is the dual of $D_0$; its groups correspond to the
parallel classes in $\mathcal{D}_0$;
\item the maps $\psi_\beta$ are described by a Latin square of order~$6$.
\end{itemize}
We applied our construction to these data, with random Latin squares generated
by the Jacobson--Matthews algorithm~\cite{jm}. A preliminary run
found $35$ non-isomorphic $2$-$(96,20,4)$ designs, all with relatively
large automorphism groups (of order at least $48$), including one
flag-transitive design (number~$1$ in the list in \cite{lpr}, with automorphism
group $2^8:((3\cdot A_6):2)$). We believe an exhaustive search is feasible.

We do not currently know whether the other three designs in \cite{lpr} can
be obtained from our construction.

\section{Examples in symplectic designs}
\label{s:symp}

The symplectic designs $S^\epsilon(n)$, for $\epsilon\in\{{+},{-}\}$ and 
$n\ge2$, form a class of symmetric designs with many remarkable properties. 
(See~\cite{cs} for a description of these designs.)
The designs $S^+(n)$ and $S^-(n)$ are complementary, and $S^\epsilon(n)$ is a
$2$-$(2^{2n},2^{n-1}(2^n+\epsilon),2^{n-1}(2^{n-1}+\epsilon))$ design.
The automorphism group of $S^\epsilon(n)$ is $2^{2n}:\Sp(2n,2)$.

We show here that $S^-(n)$ admits a
flag-transitive but point-imprimitive subgroup $2^{2n}:\GL(n,2)$, preserving
the structure of a grid on the point set.

We begin by giving two constructions of the symplectic designs.

\begin{construction}\label{con:symp1}{\rm First, the
 ``standard'' construction in terms of quadratic forms.
Let $W$ be a $2n$-dimensional vector space over $\GF(2)$, and $f$ a fixed
symplectic form on $W$. There are $2^{2n}$ quadratic forms $Q$ which polarise
to $f$, that is to say, $Q(x+y)=Q(x)+Q(y)+f(x,y)$ for all $x,y\in W$. 
These fall into two \emph{types} called $+$ and $-$, where a form of type
$\epsilon$ has $2^{n-1}(2^n+\epsilon)$ zeros. (The forms of type $+$ have
Witt index~$n$, while those of type $-$ have Witt index~$n-1$. The type as we
have defined it is $(-1)^a$, where $a$ is the \emph{Arf invariant} of $Q$.
We denote the type of the form $Q$ by $\tp(Q)$.) Now the
points of the design are the vectors in $W$, and the blocks are the quadratic
forms which polarise to $f$; the vector $w$ and quadratic form $Q$ are
incident in $S^\epsilon$ if and only if $(-1)^{Q(w)}=\epsilon\tp(Q)$.}
\end{construction}

\begin{construction}\label{con:symp2}{\rm
Now we give an alternative construction which is more explicit and easier to
work with. Take $V$ to be $n$-dimensional over
$\GF(2)$, and $V^*$ its dual space: we represent $V$ and $V^*$ by row vectors
with the action of $V^*$ on $V$ being given by $v.x=vx^\top$. Now we take
$W=V\oplus V^*$; both points and blocks of the design will be indexed by $W$,
with the point $P(v,x)$ and the block $B(w,y)$ incident if and only if
$(v+w)(x+y)^\top=0$.

First we identify this incidence structure with the symplectic design of 
$+$ type, $S^+(n)$.
The block $B(0,0)$ is incident with all points $P(v,x)$ satisfying $vx^\top=0$;
in coordinates, this is $\sum v_ix_i=0$, and $Q(v,x)=\sum v_ix_i$ is a 
quadratic form of type $+$ on $W$ (since both $V$ and $V^*$ are totally 
singular subspaces). Any other quadratic form which polarises to the same 
bilinear form can be written as
\begin{eqnarray*}
Q_{a,b}(v,x) &=&\sum v_ix_i + \sum b_iv_i^2+\sum a_ix_i^2\\ 
&=& \sum v_ix_i + \sum b_iv_i+\sum a_ix_i\\
&=& vx^\top + vb^\top + ax^\top
\end{eqnarray*}
where $a=(a_1,\ldots,a_n)\in V$ and $b=(b_1,\ldots,b_n)\in V^*$. Now
$Q_{a,b}(v,x)+ab^\top = (v+a)(x+b)^\top$. Thus the type of
this form is $+$ if $ab^\top=0$ and $-$
if $ab^\top=1$, that is, $\tp(Q_{a,b})=(-1)^{ab^\top}$. Now we see that
$(v,x)$ is incident with $Q_{a,b}$ in $S^+(n)$
\begin{center}
\begin{tabular}{ll}
if and only if &$(-1)^{Q_{a,b}(v,x)}=\tp(Q_{a,b})=(-1)^{ab^\top}$\\
if and only if &$Q_{a,b}(v,w)=ab^\top$\\
if and only if &$(v+a).(x+b)^\top=0$\\
if and only if &$P(v,x)$ is incident with $B(a,b)$
\end{tabular}
\end{center}
as required.}
\end{construction}

Now it is clear that the translation group of $W$, and the group $\GL(V)$,
acting in the natural way on $V$ and by the inverse-transpose action on $V^*$,
preserve the incidence structure and generate their semidirect product
$G=2^{2n}:\GL(n,2)$. Clearly $\GL(V)$ is imprimitive on $W$; 
the families of cosets of
$V$ and of $V^*$ are systems of imprimitivity, giving $W$ the structure of
a square grid.

It is clear that $\GL(V)$ acts in the same way on points and blocks, by
$A:P(v,x)\mapsto P(vA,x(A^{-1})^\top)$ (and similarly for blocks). In particular, it fixes the block $B(0,0)$. Now
$\GL(V)$ has five orbits on $W$, namely $\{(0,0)\}$, $\{(0,x):x\ne0\}$,
$\{(v,0):v\ne0\}$, $\{(v,x):v,x\ne0,vx^\top=0\}$, and
$\{(v,x):vx^\top=1\}$. The points corresponding to the first four orbits 
form the block $B(0,0)$, and the fifth is its complement (a block of $S^-(n)$).
So $G$ is flag-transitive on $S^-(n)$.

The block $P=B(0,0)$ of $S^-(n)$ consists of the points $P(v,x)$ for
which $vx^\top=1$. If we partition the space $W$ into cosets of $\{(0,x):
x\in V\}$, we see that this block contains no points of the subspace, and
meets every other coset in an affine hyperplane. So the ingredients for our
construction are
\begin{itemize}
\item $\mathcal{D}_0$ is the design of points and affine hyperplanes of the
affine space $\mathrm{AG}(n,2)$ based on the dual space $V^*$ of the vector
space $V$ (but note that we have an inner product which identifies $V$ with
$V^*$, so we do not have to distinguish between them);
\item $\mathcal{D}_1$ is the trivial $2$-$(2^n,2^n-1,2^n-2)$ design,
with point set $V$;
\item $\mathcal{D}_2$ is the dual of $D_0$;
\item there is a natural indexing of $\mathcal{P}_0$ by non-zero vectors
of $V$, and each block $\beta$ of $\mathcal{D}_1$ is the complement of a
single vector $v_\beta$; we take $\psi_\beta$ to map the parallel class
indexed by $v$ to the vector $v+v_\beta$. (In other words, the Latin square
used in the construction is the addition table of the vector space $V$.)
\end{itemize}

\paragraph{Remark} The group $2^{2n}:\GU(n,2)$ is flag-transitive on the
design $S^+(n)$, but is point-primitive.

\begin{question}\rm
Are there other examples of flag-transitive automorphism
groups of symmetric designs which preserve a grid structure?
\end{question}

In this connection,we have the following result.

\begin{proposition}
Let $D$ be a $2$-$(v,k,\lambda)$ design, with $v=n^2$, admitting a 
block-transitive subgroup of $S_n\times S_n$ (with the product action).
Then $n+1$ divides $\binom{k}{2}$.
\end{proposition}

\begin{pf}
As in \cite{cp}, the images of the blocks of $D$ under
$S_n\times S_n$ are the blocks of a block-transitive $2$-design. The
$2$-design condition is equivalent to the condition that the numbers of
pairs of distinct points in a block $B$ lying in each orbital of 
$S_n\times S_n$ are in the same ratio as the sizes of these orbitals, that is,
$1:1:n-1$. Hence $n+1$ divides  $\binom{k}{2}$.
\end{pf}

\section{A sporadic example}\label{s:sporadic}

In this section we construct a flag-transitive, point-imprimitive
$2$-$(1408,336,80)$ design whose automorphism group is
$2^6:((3\cdot\mathrm{M}_{22}):2)$. We describe briefly the way this design
might be produced by our general construction, but for verification purposes
it is more convenient to give a more direct construction of the design.

The ingredients are as follows.
\begin{enumerate}
\item $\mathcal{D}_0$ is the design of points and planes in $3$-dimensional
affine space over $\GF(4)$. It is a $2$-$(64,16,5)$ design with $21$ parallel
classes of size~$4$. 
\item $\mathcal{D}_1$ is the trivial symmetric $2$-$(22,21,20)$ design.
\item $\mathcal{D}_2$ is the dual of $\mathcal{D}_0$. That is, its points
are the $84$ planes of $\AG(3,4)$, partitioned into $21$ groups (parallel
classes) of size $4$; for each point of the affine space, there is a block
of $\mathcal{D}_2$ consisting of the $21$ planes containing that point, with
one plane from each parallel class. Two non-parallel planes meet in four
points.
\item The construction of the bijections $\psi_\beta$ uses a $6$-dimensional
affine space admitting the group $2^{12}:(3\cdot\mathrm{M}_{22})$. It is
somewhat intricate: this group acts on a set of size $22$ (which we identify
with the point set of $\mathcal{D}_1$); the stabiliser of a point has a subgroup
$2^6:(3\cdot\mathrm{M}_{21})$ which acts on the design $\mathcal{D}_0$.
The latter group acts on $21$ points as the stabiliser of a point in
$\mathrm{M}_{22}$, and we might use this to construct the bijection.
Instead, we proceed by a group-theoretic construction as follows.
\end{enumerate}

The group $G=3\cdot\mathrm{M}_{22}$ (a triple cover of the Mathieu group
$\mathrm{M}_{22}$) has a $6$-dimensional irreducible module $V$ of degree~$6$
over $\GF(4)$. (Matrices generating this group can be found in the on-line
Atlas of Finite Group Representations~\cite{atlas}.) The module is reducible
for the subgroup $H=3\cdot\mathrm{M}_{21}\cong\SL(3,4)$, which has a
$3$-dimensional submodule $W$.
The subgroup $WH$ has index $4^3\cdot22=1408$ in $VG$, and the action
of $VG$ on the cosets of $WH$ is imprimitive with $22$ blocks of size~$64$
(corresponding to the cosets of $VH$).

Now the stabiliser of a point in this group has orbit lengths $1$, $63$
(these two form a block), $336$ and $1008$ (these meet each remaining block
in $16$ and $48$ points respectively). The translates of the orbit of size
$336$ form a symmetric $2$-$(1408,336,80)$ design, which is obviously
flag-transitive and point-imprimitive. (The calculations to verify this
assertion were done using \textsf{GAP}~\cite{gap}. Most of the time taken by
the computation is in constructing the action of $VG$ on the cosets of $WH$.)
\textsf{GAP} also verifies that the group used in the construction is a 
subgroup of index~$2$ in the full automorphism group of the design, which
is itself point-imprimitive. (The outer automorphism of order $2$ restricts
to a field automorphism in $\Sigma\mathrm{L}(3,4)$, and so does not show up
in the linear action of $\mathrm{SL}(3,4)$ on the module.)

It would be useful to identify explicitly the bijections $\psi_\beta$ (or
equivalently, the appropriate Latin square of order $22$) so as to give an
alternative construction of the design. Note, however, that there are many
different Latin squares of order $22$, and so we expect to find many
non-isomorphic symmetric $2$-$(1408,336,80)$ designs.

\section{Open problems}

To conclude, here are some open problems.

\begin{question}\rm
Does every flag-transitive, point-imprimitive symmetric $2$-design arise from
our construction? The $2$-$(96,20,4)$ designs are the obvious test case for
this question.
\end{question}

\begin{question}\rm
In all the examples of our construction of a symmetric $2$-design, the
design $\mathcal{D}_2$ is the dual of $\mathcal{D}_0$, and has the property
that every block meets every group in a unique point. Does this necessarily
happen? Note that the points of $\mathcal{D}_2$ are identified with the blocks
of $\mathcal{D}_0$, and the partition of the point set into groups corresponds
to the given resolution of $\mathcal{D}_0$. In all our examples, the equality
$k_2=r$ holds, and this is equivalent to the property that every block of
$\mathcal{D}_2$ meets every group.
\end{question}

\begin{question}\rm
There is an unexplained correspondence between our (symplectic and sporadic)
examples, and some $2$-transitive symmetric $2$-designs which are extendable
to $3$-designs admitting transitive extensions of automorphism groups of the
smaller designs. 

In the sporadic example, $\mathcal{D}_0$ is affine space $\mathrm{AG}(3,4)$,
which induces the projective plane $\mathrm{PG}(2,4)$ on the points at
infinity; the lines of this plane correspond to the parallel classes of blocks
of $\mathcal{D}_0$ (forming the set $\mathcal{P}_0$). Now this plane has an
extension to a $3$-$(22,6,1)$ design with automorphism group $M_{22}:2$, which
is the group we have acting on the points of $\mathcal{D}_1$. In a similar
way, in the symplectic examples, $\mathcal{D}_0$ is the affine space
$\mathrm{AG}(n,2)$, whose hyperplane at infinity carries $\mathrm{PG}(n-1,2)$;
this design can be extended back to $\mathrm{AG}(n,2)$ admitting the affine
group.

Note however that the designs here (the Witt $3$-design for $M_{22}:2$ and
the affine space) play no role in the construction of our symmetric designs.

The problem is to explain what is going on here!
\end{question}

The computations in this paper were performed using \textsf{GAP}~\cite{gap}
and its \textsf{DESIGN} package~\cite{design}.

\end{document}